\documentclass [12pt,a4paper]{amsart}

\usepackage{stmaryrd}
\usepackage{textcomp}
\usepackage{amsmath}
\usepackage{cases}
\usepackage{amscd}
\usepackage{epsfig}
\usepackage{graphicx}

\usepackage{color}
\usepackage{verbatim}
\usepackage{amssymb}
\usepackage{amsfonts}
\usepackage{amsbsy}
\usepackage{amsthm}
\usepackage{amstext}
\usepackage{amsopn}

\newcommand{\HC}{{\textbf{H}_{\mathbb{C}}^2}}
\newcommand{\PU}{\textrm{PU}}
\newcommand{\SU}{\textrm{SU}}
\newcommand{\tr}{\textrm{tr}}

\renewcommand{\Re}{\textrm{Re}}
\renewcommand{\H}{\textrm{H}}
\newcommand{\C}{\mathbb{C}}
\newcommand{\n}{\textbf{n}}
\renewcommand{\a}{\textbf{a}}
\renewcommand{\b}{\textbf{b}}

\newtheorem{thm}{Theorem}[section]

\newtheorem{prop}[thm]{Proposition}

\newtheorem{rem}[thm]{Remark}

\usepackage{fancyhdr}
\begin{document}
\title{ $\mathbb{C}$-Fuchsian subgroups of some non-arithmetic lattices}
\author{Li-Jie Sun}
%\date\today
\thanks{E-mail: ljsun@yamaguchi-u.ac.jp}

\address {Department of applied science, Yamaguchi University
2-16-1 Tokiwadai, Ube 755-8611,
Japan}
\email{ljsun@yamaguchi-u.ac.jp}

\maketitle
\begin{abstract}

We give a general procedure to analyze the structure for certain $\mathbb{C}$-Fuchsian subgroups of some non-arithmetic lattices. We also show their presentations and describe their fundamental domains which lie in a complex geodesic, a set homeomorphic to the unit disk.

%In \cite{DPP2}, the authors stated a process to produce fundamental domains for complex hyperbolic triangle groups which are non-arithmetic lattice in $\PU(2, 1)$. Starting from that point of view, we find some Fuchsian subgroups of complex hyperbolic triangle groups and draw the corresponding fundamental domains.
\end{abstract}
\section{Introduction}
Suppose that ${\rm H}$ is a Hermitian form of signature $(2, 1)$ on $\C^3.$ Then the projective unitary Lie group $\PU(2, 1)$ of 
${\rm H}$ contains two conjugacy classes of connected Lie subgroups, each of which is locally isomorphic to 
$\rm{PSL}(2, \mathbb{R}).$ The subgroups in one class are conjugate to ${\rm PSU}(1, 1)$, and preserve a complex line for the projective action of $\PU(2, 1)$ on the projective plane ${\bf P}_\mathbb{C}^2$. The subgroups in the other class are conjugate to 
${\rm PO}(2, 1)$, and preserve a totally real Lagrangian plane. If $\Gamma$ is a discrete subgroup of $\PU(2, 1),$ the intersections of 
$\Gamma$ with the connected Lie subgroups locally isomorphic to ${\rm PSL}(2, \mathbb{R})$ are its Fuchsian subgroups. The Fuchsian subgroups fixing a complex line are called $\mathbb{C}$-Fuchsian subgroups. See Section \ref{Sec:pre} for more details.

%Some of  the groups are not contained (up to commensurability) in the list of Deligne-Mostow lattices and Thurston lattices. By the construction of an explicit fundamental domain for each group , they prove that the group is a lattice. There are general methods to produce fundamental domains in $\HC$, such as Dirichlet domains, however the combinatorial intersections of two bisectors are quite complicated, see \cite{Gol}. Richard Schwartz constructed the fundamental domains in \cite{Sch4} and the general procedure of the combinatorial constructions of the domain is applied to a wide class of complex hyperbolic triangle groups. From the appendix in \cite{DPP2}, we could see the fundamental domains of 9 kinds of complex hyperbolic triangle groups. The $\mathbb{C-}$Fuchsian subgroup in the lattice $\Gamma$ is a subgroup which fix a complex geodesic $L$ (see the definition in Section 2.4).

Fuchsian subgroups have remarkable geometrical properties and they are interesting on their own, see for instance \cite{PaPl,PPau}. 
% The construction of fundamental domains is much more complicated since there are no totally geodesic real hypersurfaces,  see section \ref{Subsec:tga}. The complex hyperbolic triangle groups as an important type of subgroups and candidates of complex hyperbolic lattices in $\PU(2, 1)$ has  been studied by many mathematicians.  
They also play an important role in complex hyperbolic space. Deraux \cite{Der} proved that the discrete deformation of some $\mathbb{R}$-Fuchsian triangle group in $\PU(2, 1)$ is a cocompact arithmetic lattice (a lattice in $\PU(2, 1)$ is a discrete group with finite covolume).
There also have been significant developments on $\C$-Fuchsian subgroups. To this direction,
let $S$ be a hyperbolic surface. Gusevskii-Parker \cite{GuPa} studied the deformation space of a $\C$-Fuchsian representation $\pi_1(S)\to$ {\bf{Isom}}$(\HC)$ by formulating and proving Poincar\'e's polyhedron theorem for one special class of polyhedra in complex hyperbolic plane.
Furthermore, Stover \cite{Sto} proved that if $\Gamma$ is a complex hyperbolic lattice containing a complex reflection, then $\Gamma$ contains a $\C$-Fuchsian subgroup stabilising the complex geodesic fixed by the reflection. However, it is usually difficult to get an explicit description of such $\C$-Fuchsian subgroups from the complex hyperbolic lattice.
%prove directly that a given group, for instance generated by a number of generators, acts discretely on complex hyperbolic plane.
In the present paper, we wish to identify the structures of the $\mathbb{C}$-Fuchsian subgroups (arising as stabilisers of the complex geodesics fixed by reflections) in complex hyperbolic lattices, mainly by applying Poincar\'e's polygon theorem. In this way, we also illustrate their actions on the fixing complex geodesic $L.$
In other words, we get a more explicit version of Stover's result. The study on the structure of the stabilisers of the complex lines is useful in the study of complex hyperbolic lattices using algebraic geometry (see for example \cite{Dera}), and is useful for considering lattices from the point of view of hybrids (see \cite{Jose}).
% see {\textcolor{red}{Theorem \ref{thm:lat}. }maybe cancelled}

In \cite{DPP, DPP2}, Deraux, Parker and Paupert considered a family of groups which produce all currently known examples of non-arithmetic lattices in $\PU(2, 1)$. Each of such groups is a complex hyperbolic triangle group generated by three complex reflections of the same order $p$ ($p \geq 2$). They prove the discreteness by constructing an explicit fundamental domain for each group, and show that the geometric realisation gives an embedding of the combinatorial fundamental domain into the topological closure of complex hyperbolic plane $\overline{{\bf {H}}_\mathbb{C}^2}.$ 
 In particular, the authors listed the side (codimension-1) representatives of the fundamental domains for the sporadic triangle groups (see Section \ref{Subsec:spo}) and Thompson triangle groups (see Section \ref{Subsec:Tho}), also gave the natural representation for each group.  
 
In this paper, our goal is to identify the $\mathbb{C}$-Fuchsian subgroups of the sporadic triangle groups (subgroups of equilateral triangle groups) and Thompson triangle groups (subgroups of non-equilateral triangle groups), which appeared in \cite{DPP2}.
We consider the equilateral triangle groups which are generated by three complex reflections $R_1,\, R_2,\, R_3$ with the property that there exists a complex hyperbolic isometry $J$ of order 3 such that $R_{j+1}=JR_{j}J^{-1}$ (the indices taken by mod 3).  The equilateral triangle groups then can be parameterised by the order $p$ of generators and the complex parameter
\[
\tau=\tr(R_1J).
\]
We denote the sporadic triangle groups by $\mathcal{S}(p, \tau).$ See details in Section \ref{Subsec:spo}.

Our main theorem is the following:
 \begin{thm}\label{thm:main}
 Let $R_1, R_2, R_3$ be three complex reflections of order $p$ in ${\rm SU}(2, 1)$ so that $R_i$ fixes a complex geodesic $L_i,$ $i=1, 2, 3.$ Suppose that $R_1, R_2, R_3$ is the generating set for $\mathcal{S}(p, \tau)$. Then there exist $\mathbb{C}$-Fuchsian subgroups  fixing complex geodesic $L_1$ which have the following structure according to $(\tau, p):$
\begin{enumerate}
\item[{(i)}] $\tau=-1+i\sqrt{2}, p=3, 4, 6:$ 
$$\big\langle   (1\bar{3}23)^2, (13)^3, (12)^3, (123\bar{2})^2, (1232\bar{3}\bar{2})^3 (123\bar{2})^2 (12)^3 \big\rangle;$$
\item[{(ii)}]  $\tau=-\frac{1+i\sqrt{7}}{2}, p=3, 4, 5, 6, 8, 12:$ 
$$\big\langle(12)^2,(13)^2,23\bar{2}P^2\big\rangle,$$
where $P=R_1 J;$
\item[{(iii)}] $\tau=\frac{1+\sqrt{5}}{2}, p=3, 4, 5, 10:$ 
$$\big\langle1\bar{3}\bar{2}323, 1312\bar{1}\bar{3}, (1\bar{3}23)^31\bar{3}\bar{2}323\big\rangle.$$
\end{enumerate}
 \end{thm}
 Here we just write $1\bar{3}23$ to denote $R_1R_3^{-1}R_2R_3$ (see Section \ref{Subsec:cr}), etc. 
 %By studying the combinatorics of {{invariant shells }}  (see section 4.1 in \cite{DPP2}) of the fundamental domain for the complex hyperbolic lattice, we distinguish a polygon with finitely many sides lying in the chosen complex geodesic which can be matched by side pairing transformations. Those side pairing transformations are exactly the generators of the $\C$-Fuchsian subgroups as showed in Theorem \ref{thm:main}. We also show the explicit fundamental domains for the $\C$-Fuchsian subgroups as stated above. 
 %{{\textcolor{red}{We also remark that the full stabiliser of $L$ denoted by {\bf Stab}$_{\PU(2, 1)}(L)$ is a central extension of {\bf Stab}$_{\mathcal{S}(p, \tau)}(L)$.}}
 Throughout this paper, we always investigate the $\C$-Fuchsian subgroups fixing a complex geodesic $L_1$. One should note that there naturally exist $\C$-Fuchsian subgroups fixing other complex geodesics in the complex hyperbolic lattice under consideration. For example, we  could get
a $\C$-Fuchsian subgroup  in $\mathcal{S}(3, -1+i\sqrt{2})$ stabilising the complex geodesic $L_{2}$  fixed by the complex reflection $R_2$ which is identified with $JR_1J^{-1}$:
$$\big\langle   J(1\bar{3}23)^2J^{-1}, J(13)^3J^{-1}, J(12)^3J^{-1}, J(123\bar{2})^2J^{-1}, J(1232\bar{3}\bar{2})^3 (123\bar{2})^2 (12)^3J^{-1} \big\rangle.$$

 %On the other hand, let $g\in\PU(2, 1)$ and $L$ be a complex geodesic. For two distinct points $z_1, z_2\in L,$ A simple fact is that $g(z_1), g(z_2)\in L$ yields to $g(L)=L.$ Actually, two distinct points in $\overline\HC$ can span one and only one complex line (see Section \ref{Subsec:tgs}), i.e. lie in one and only one complex geodesic simultaneously. 
 %we If $g$ maps complex geodesic $L$ to another complex geodesic $L^\prime$. 
% Therefore we can get that $g(L)=L.$
% From this point of view, we can also get that each side paring transformations does fix the choosing complex geodesic.

 %The polygon is actually the fundamental domain of $\C$-Fuchsian subgroups, i.e.  Each lattice is of  sporadic triangle group or Thompson group, see details in Section \ref{Sec:ST}.{\textcolor{red}{state why it might be interesting to describe such subgroups }. motived by the above 
 %The purpose of this paper is to consider the  $\mathbb{C}-$ Fuchsian subgroups in the triangle lattices, and show their structure. It is important for imaging  the combinatorial of fundamental domain. SU(1, 1)lattice in SU(2, 1)}
In \cite{DPP2}, the authors build blocks of the fundamental domains bounded by spherical shells that surround the fixed point of $P=R_1J$ for the lattices of equilateral triangle groups type or $Q=R_1R_2R_3$ in the non-equilateral case. A spherical shell here means that the corresponding cell complex is an embedded copy of $S^3,$ which bounds a well-defined 4-ball. Surrounding a point just means that the point is in the ball component of the complement of that copy of $S^3.$ The basic building blocks for their fundamental domains are pyramids (for example, see Figure 1) in bisectors. They finally list side (codimension-1) representatives for each $P$-orbit of sides (or $Q$-orbit in the non-equilateral case), and one side for each pair of opposite sides which means paired in the sense of the Poincar\'e polyhedron theorem, see Appendix in \cite{DPP2}. In the present paper, our general procedure to distinguish $\mathbb{C}$-Fuchsian subgroups is as follows: We firstly focus on the pyramids of the side representatives of the fundamental domain for the non-arithmetic lattices; secondly, for each lattice, we force the side representatives to have the same base $L_1$ and obtain a polygon lying in the complex geodesic $L_1$; finally we prove that the polygon is a fundamental domain of some subgroup of the lattice. Actually the polygon can be matched by side pairing transformations, which are exactly the generators of the $\mathbb{C}$-Fuchsian subgroups as showed in Theorem \ref{thm:main}. We also give the natural presentation for each $\mathbb{C}$-Fuchsian subgroup among the proof.

The paper is arranged as follows. Section 2 contains background material about complex hyperbolic plane $\HC$, totally geodesic subspaces and complex reflection. In Section 3 we recall the normalisation of two kinds of complex hyperbolic triangle groups in $\PU(2, 1)$: equilateral triangle groups and non-equilateral triangle groups, in which we will clarify the $\C$-Fuchsian subgroups. 
%which are the stabilisers of the complex geodesics. 
In Section 4, we mainly state and prove our theorems, including describing the fundamental domains of certain $\C-$Fuchsian subgroups. 
% From this point of view we show that they are lattice which can be embedded in $\SU(1, 1)$.

\section{Preliminaries}\label{Sec:pre}
%\subsection{Real hyperbolic plane}
The material for this section is standard. The reader may refer to \cite{Gol} for more details.
\subsection{Complex hyperbolic plane} We use $\C^{2, 1}$ to denote $\C^3$ equipped with a Hermitian form of signature $(2, 1)$. If we assume that {\bf P} is the canonical projectivisation from $\C^{2, 1}$ to 
${\bf P}_\mathbb{C}^2$ and suppose that the Hermitian form of signature $(2, 1)$ to be ${\textrm H}$, then the {\textit{complex hyperbolic plane}} $\HC$ can be defined as follows:
\[
\HC:={\bf P}\{{\bf z}\in \C^{2, 1}: \langle{\bf z, z} \rangle= {\bf\bar{z}^{t}}\H{\bf z}<0\}.
\]
Correspondingly, the boundary $\partial\HC$ of complex hyperbolic plane is 
\[
\partial\HC:={\bf P}\{{\bf z}\in \mathbb{C}^3: \langle {\bf z, z} \rangle= {\bf\bar{z}}^{t}\H{\bf z}=0\}.
\]
There exists a natural action of the unitary group ${\rm U}(2, 1)$ of the Hermitian from on $\HC.$
The automorphism group of $\HC$ is then $\PU(2, 1)$, the projectivisation of ${\rm U}(2, 1)$. In particular, ${\rm SU}(2, 1)$ is the subgroup of  ${\rm U}(2, 1)$ with the determinant of each element being $1$, which is the three fold cover of the projection group $\PU(2, 1)$.

Let $z$ and $w$ be points in $\HC$ corresponding to vectors ${\bf {z, w}}\in\C^{2, 1}.$ Then the {\textit{Bergman metric}} $\rho$ on $\HC$ is given by the following distance formula:
$$\cosh^2 \left(\frac{\rho(z, w)}{2}\right)=\bf \frac{\langle z,w\rangle\langle w, z\rangle}{\langle z, z\rangle\langle w, w\rangle}.$$
If we choose the Hermitian form of signature $(2, 1)$ as follows
$$\langle {\bf{z, w}}\rangle=z_1\overline{w_1}+z_2\overline{w_2}-z_3\overline{w_3},$$
with ${\bf{z}}=[z_1, z_2, z_3]^t,\, {\bf{w}}=[w_1, w_2, w_3]^t,$ then the complex hyperbolic plane $\HC$ can be described in the affine chart $z_3\neq0$ as the unit ball in $\mathbb{C}^2$ endowed with the unique K\"ahler metric invariant under all biholomorphisms of the ball. The metric is symmetric and has non-constant negative real sectional curvature but pinched between $-1$ and $-1/4$. We normalise its holomorphic sectional curvature to be $-1$.

An automorphisms of $\HC$ is said to be elliptic if it fixes at least one point of $\HC$,
parabolic if it fixes exactly one point of $\partial\HC$, loxodromic if it fixes exactly two points of $\partial\HC$. Throughout this paper, we freely use the classification of automorphisms of $\HC$ into regular elliptic, complex reflection, ellipto-parabolic, unipotent parabolic and loxodromic, e.g., an automorphism is regular elliptic if and only if it has a fixed point in $\HC$ and has distinct eigenvalues. We refer to Section 6.2 of \cite{Gol} for the details.
%Throughout this paper, we freely use the classification of automorphisms of $\HC$ into regular elliptic, complex reflection, ellipto-parabolic, unipotent parabolic and loxodromic elements, e.g., an automorphism is regular elliptic if and only if it has a fixed point in $\HC$ and has distinct eigenvalues. We refer to section 6.2 of \cite{Gol} for the details.

\subsection{Totally geodesic subspaces}\label{Subsec:tgs} Given two points $z$ and $w$ in $\overline{\HC}:=\HC\cup\partial\HC,$ with lifts $\bf{z, w}$ to $\mathbb{C}^{2, 1}$ respectively, the complex span of $\bf{z}$ and ${\bf{w}}$ projects to a complex projective line in ${\bf P}_\mathbb{C}^2$. The intersection of a complex projective line with $\HC$ is called a {\textit{complex geodesic}} $L$ (homeomorphic to an open 2-dimensional disk), which can be 
simply obtained by taking the intersection of orthogonal complement of a positive vector $\bf{n}$ with $\HC$, i.e., $$L={\bf P}\{{\bf z}\in\C^{2, 1}: \langle {\bf z, n}\rangle=0\}\cap\HC.$$ We refer to $\bf{n}$ as a {\textit{polar vector}} to  $L$.

A maximal totally geodesic subspace in $\HC$ can only be one of the following:
\begin{enumerate}
\item[{(i)}] A complex geodesic, which is an isometrically embedded copy of $\textbf{H}^{1}_{\mathbb{C}}$. It has the Poincar\'e model of hyperbolic geometry with constant curvature $-1$;
\item[{(ii)}] A totally real Lagrangian plane, which is an isometrically embedded copy of $\textbf{H}^{2}_{\mathbb{R}}$. It has the Beltrami-Klein projective model with constant curvature $-1/4$.
\end{enumerate}

%The intersection of a complex geodesic $L$ with $\partial \HC$ is called a \emph{$\mathbb{C}$-circle}. Correspondingly, the intersection of a totally real Lagrangian plane with $\partial\HC$ is called an \emph{$\mathbb{R}$-circle}.
%For more details we refer for instance to \cite{Gol}.

\subsection{Complex reflection}\label{Subsec:cr} Suppose that the polar vector of a complex geodesic $L_{1}$ is $\bf n_{1}$. We consider the complex reflection $R_{1}$ in the complex geodesic $L_{1}$ which is of order $p$, i.e., complex reflection $R_{1}$ in $\rm{U}(2, 1)$ maps $\bf n_{1}$ to $e^{i\phi}\bf n_{1}$, where $\phi=2\pi/p$. Throughout this paper, we assume that $p\in\mathbb{Z}$ and $p\geq2$. We take one lift of $R_{1}$ to a matrix in $\SU(2, 1)$ and  write the map here with the same symbol:
\begin{equation}\label{ref:l}
R_{1}({\bf z})=e^{-\frac{i\phi}{3}}{\bf z}+(e^{\frac{2i\phi}{3}}-e^{-\frac{i\phi}{3}})\frac{\langle {\bf z}, {\bf {\bf{n_1}}}\rangle}{\langle {\bf {\bf{n_1}}}, {\bf {\bf{n_1}}}\rangle}{\bf {\bf{n_1}}}.
\end{equation}
%{\textcolor{red}{matrix representation} }\\
A basic fact is that any complex reflection is an element of $\PU(2, 1)$. We will restrict to the complex hyperbolic triangle groups generated by three complex reflections with the same order $p$
$(p\geq2).$ 
In order to avoid tedious notation, we denote the three generators $R_1,\, R_2,\, R_3$ of complex hyperbolic triangle groups simply by $1,\, 2,\, 3$. Unless otherwise stated,  in what follows we also denote their inverse by $\bar{1},\, \bar{2},\, \bar{3}.$ In this way, we just write $1\bar{3}23$ to denote $R_1R_3^{-1}R_2R_3$, etc.

We recall the definition for {\textit{braid relation}} between group elements (see Section 2.2 of \cite{Mos}). Let $G$ be a group  and $a, b\in G$. Then we will say that $a, b$ satisfy a braid relation of length $n\in \mathbb{Z}_{+}$ if 
$$(ab)^{n/2}=(ba)^{n/2},$$
where powers mean that the corresponding alternating product of $a$ and $b$ should have $n$ factors. We denote the braid length $n$ of $a, b$ by ${\rm{br}}_n(a, b)$. For example, ${\rm{br}}_3(a, b)$ means that $aba=bab.$

Let $A$ and $B$ be two complex reflections in distinct complex geodesics $L_A$ and $L_B$ respectively, which correspond to polar vectors 
$\n_A$ and $\n_B.$ The cross-product ${\bf{z}}:=\n_A\boxtimes\n_B$ is defined as 
$${\bf{z}}=({\overline{\n_A}}^t{\rm H})\times({\overline{\n_B}}^t{\rm H}).$$

Then three possibilities arise (see Section 3.3.2 in Goldman \cite{Gol}):
\begin{enumerate}
\item{} $\bf{z}$ is negative, namely $\langle \bf{z},  \bf{z}\rangle<0$. In this case $L_A$ and $L_B$ intersect in $\bf{P}(\bf{z})\in\HC$ corresponding to the negative vector $\bf{z}$;
\item{}$\bf{z}$ is null, namely $\langle \bf{z},  \bf{z}\rangle=0$. In this case $L_A$ and $L_B$ are asymptotic at the point $\bf{P}(\bf{z})\in\partial\HC$;
\item{}$\bf{z}$ is positive, namely $\langle \bf{z},  \bf{z}\rangle>0$. In this case $L_A$ and $L_B$ are ultraparallel, that is they are disjoint and have a common
orthogonal complex geodesic, which is polar to $\bf{z}$.
\end{enumerate}

%We will mainly restrict to the intersection case in Section 4.

%\subsection{Fuchsian groups} A discrete group $G\subset\SU(2, 1)$ is non-elementary if the limit set of $G$ has more than two points. Following \cite{CG1}, for any non-elementary discrete subgroup $G$ of $\SU(2, 1)$, 
%\begin{enumerate}
%\item[{(i)}] $G$ is called $\bf{\C-Fuchsian}$ if it fixes a complex geodesic;
%\item[{(ii)}] $G$ is called $\bf{\mathbb{R}-Fuchsian}$ if it fixes a totally real Lagrangian plane.
%\end{enumerate}
%Throughout this paper, we restrict to $\C - $Fuchsian groups, and simply write as Fuchsian groups.
%$G$ is called $\bf{\C-Fuchsian}$ if it is the extension of a subgroup of $\textrm{SU}(1, 1)$;
%$G$ is called $\bf{\mathbb{R}-Fuchsian}$ if it is the extension of a subgroup of $\textrm{SO}(2, 1).$ In other words, 

\section{Sporadic triangle groups and Thompson triangle groups}\label{Sec:ST}
In this section we review sporadic triangle groups  (Section \ref{Subsec:spo}) and Thompson triangle groups (Section \ref{Subsec:Tho}), which we will mainly study in Section 4. For these two kinds of complex hyperbolic triangle groups, we refer for instance to \cite{DPP2, PP, Thom} for the details.

\noindent \subsection{ {\bf Equilateral triangle groups}.}\label{Subsec:spo}
Recall from the introduction that an equilateral triangle group can be generated by a complex reflections $R_1$ and a complex hyperbolic isometry $J$ of order 3. Let 
$$R_{2}=JR_{1}J^{-1},\quad R_{3}=JR_{2}J^{-1}$$
  The equilateral triangle groups then can be parameterised by the order $p$ of generators and the complex parameter
\[
\tau=\tr(R_1J).
\]
It is difficult to give the conditions of $p$ with $\tau$ so that the equilateral triangle group is a lattice, or at least discrete. However, the pairwise product of generators should be non-loxodromic (see \cite{Sch4}). 
%Also Parker investigated some candidates of $\mathcal{S}(p, \tau)$ to be discrete in \cite{Par3}. 
This shows that there are two continuous families satisfying that $R_1 J$ and $R_1 R_2$ are elliptic
\[
\tau=-e^{i\psi/3},\qquad  ~ \qquad \tau=e^{i\psi/6}\cdot2\cos(\psi/2), 
\]
where $\psi$ are rational multiples of $\pi$. 
These two families correspond to Mostow groups or certain subgroups of Mostow groups. For such groups, the list of lattices can be obtained from the work of Deligne-Mostow (see \cite{Par3, PP}). 
There are still lattice candidates not lying on these two families. In \cite{DPP2} the authors  show that the equilateral triangle groups for some values of $\tau=\tr(R_1J)$ indeed contain lattices, of which the explicit values of $\tau$ and $p$ are in Table 1. They are called {\textit{sporadic triangle groups}}.  We denote the corresponding group by $\mathcal{S}(p, \tau)$. 

Note that the list here is given up to complex conjugation and multiplication by a cube root of unity. 
In Section 4, we will give the analysis on $\C$-Fuchsian subgroups of complex hyperbolic lattices $\mathcal{S}(p, \tau)$
for $\tau=\tau_1,\, \tau_2,\, \tau_{4}$. %It is plausible to consider that one could also identify the $\C$-Fuchsian subgroups for $\tau_3;$ however, this has not been achieved by our present method.
 %\vskip 1\baselineskip
 \begin{center}
\begin{table}[h]\label{spo}
\begin{tabular}{|c|c|}
\hline
$\tau$ & $p$ \\
\hline
$\tau_1=-1+i\sqrt{2}$&$3, 4, 6$\\\hline
$\tau_2=-(1+i\sqrt{7})/2$&$3, 4, 5, 6, 8, 12$\\\hline
$\tau_3=e^{-\pi i/9}(-e^{-2\pi i/3}-(1-\sqrt{5})/2)$&$2, 3, 4$\\\hline
$\tau_{4}=(1+\sqrt{5})/2$&$3, 4, 5, 10$\\\hline
\end{tabular}
\vskip0.5\baselineskip
\caption{Values of $p, \tau$ such that $\mathcal{S}(p, \tau)$ are lattices.}
\end{table}
\end{center}
\noindent \subsection{ {\bf Non-equilateral triangle groups}.} \label{Subsec:Tho}
In this section, we review notation for the non-equilateral triangle groups which come from Thompson's thesis \cite{Thom}. 
%Thompson mainly investigated each generators of triangle groups is just of order $2$. 
They can be parameterised by a triple of complex numbers $\rho,\, \sigma,\, \tau$. The three numbers will be all equal to $\tau$ as above in the case of equilateral triangle. In the same fashion, we assume that the generators are of order $p$, $u=e^{2 \pi i/3p}$ and the Hermitian form is
\[
\H=\left(\begin{matrix}
  \alpha &\quad \beta_1 &\quad \bar{\beta_3}\\
  \bar{\beta_1}&\quad \alpha&\quad \beta_2\\
  \beta_3&\quad \bar{\beta_2}&\quad \alpha
\end{matrix}\right),
\]
where $\alpha=2-u^3-\bar{u}^3$, $\beta_1=(\bar{u}^2-u)\rho$, $\beta_2=(\bar{u}^2-u)\sigma$, $\beta_3=(\bar{u}^2-u)\tau$ and
\begin{small}
\[
\rho=(u^2-\bar{u}){\bf\frac{\langle {\bf{n_2}}, {\bf{n_1}}\rangle}{\|{\bf{n_2}}\|\| {\bf{n_1}}\|}},\quad
\sigma=(u^2-\bar{u}){\bf\frac{\langle {\bf{n_3}}, {\bf{n_2}}\rangle}{\|{\bf{n_3}}\|\| {\bf{n_2}}\|}},\quad
\tau=(u^2-\bar{u}){\bf\frac{\langle {\bf{n_1}}, {\bf{n_3}}\rangle}{\|{\bf{n_1}}\|\| {\bf{n_3}}\|}}.
\]
\end{small}
The generators which preserve the above Hermitian form $\H$ are given by
\begin{small}
\begin{equation}\label{Thom:mat}
\quad
R_1=\left(\begin{matrix}
  u^2 &\quad \rho &\quad -u\bar{\tau}\\
  0&\quad \bar{u}&\quad 0\\
  0&\quad 0&\quad \bar{u}
\end{matrix}\right),\quad
R_2=\left(\begin{matrix}
  \bar{u} &\quad 0 &\quad 0\\
  -u\bar{\rho}&\quad u^2&\quad \sigma\\
  0&\quad 0&\quad \bar{u}
\end{matrix}\right),\quad
R_3=\left(\begin{matrix}
  \bar{u} &\quad 0 &\quad 0\\
  0&\quad \bar{u}&\quad 0\\
  \tau&\quad -u\bar{\sigma}&\quad u^2
\end{matrix}\right).
\end{equation}
\end{small}
The elements  $R_1, R_2, R_3$ are determined up to conjugacy by $|\rho|, |\sigma|, |\tau|$ and $\arg(\rho\sigma\tau)$, see \cite{DPP, PP}.
\begin{table}[h]\label{Tab:tom}
\begin{tabular}{|c|c|c|c|c|c|c|}
\hline
 &$a$\quad $b$\quad $c$\quad $d$&$o(123)$&$\rho$&$\sigma$&$\tau$&lattices for $p$ \\
\hline
$S_2$&$3$\quad $3$\quad $4$\quad $5$&  $5$& $1+\frac{1+\sqrt{5}}{2}e^{2\pi i/3}$& $1$& $1$& $3,4,5$\\\hline
${E_2}$&$3$\quad $4$\quad $4$\quad $4$&  $6$& $\sqrt{2}$& $e^{-2\pi i/3}$& $\sqrt{2}$& $3,4,6,12$\\\hline
$H_1$&$3$\quad $3$\quad $4$\quad $7$&  $42$& $\frac{-1+i\sqrt{7}}{2}$& $e^{-4\pi i/7}$& $e^{-4\pi i/7}$& $2, -7$\\\hline
$H_2$&$3$\quad $3$\quad $5$\quad $5$&  $15$& $-1-e^{-2\pi i/5}$& $e^{4\pi i/5}$& $e^{4\pi i/5}$& $2, 3,5,10, -5$\\\hline
\end{tabular}
\vskip0.5\baselineskip
\caption{Lists of parameters of some lattices in Thompson triangle groups. The negative values of $p$ correspond to the conjugate values of parameters of Thompson triangle groups.}
\end{table}
Suppose that the order of $23,\, 31,\, 12$ and $1\bar{3}23$ are $a,\, b,\, c,\, d$ respectively. We write $(a, b, c; d)$ for the groups generated by complex reflections in a triangle with angles $\pi/a, \pi/b, \pi/c$ satisfying that the order of $1\bar{3}23$ is $d$. Note that here $a,\, b,\, c\geq 3$ because the $(2, b, c)$ triangle groups are rigid in $\PU(2, 1)$. In Table 2, we list only some values of $\rho, \sigma, \tau$ which correspond to lattices. For the construction of the fundamental domain of these lattices, we refer to \cite{DPP2} for further details. We give the explicit structures of the $\C$-Fuchsian subgroups stabilising the complex geodesic $L_1$ in Thompson triangle groups $S_2$ and $E_2$ after Remark \ref{rem:Thom} 
%Theorem \ref{thm:thom1} and Theorem \ref{thm:thom2}. %Also one could possibly identify the $\C$-Fuchsian subgroups for $H_1$, $H_2;$ however, this has not been achieved by our present method.

It is plausible to consider that one could also identify the $\C$-Fuchsian subgroups for 
$\mathcal{S}(p, \tau_3)\, (p=2, 3, 4),$ $H_1$, $H_2;$ however, it has not been achieved by our present method. The main difficulty is to find an appropriate polygon and the transformations which pair the sides of the polygons lying the complex geodesics under consideration.
%{ \textcolor{red}{Conjecture for $\sigma_5$ and some words for H1and H2.}}

\section{$\C$-Fuchsian subgroups and their explicit Fundamental domains}\label{Sec:thm}
Let us firstly recall the Poincar\'e polygon theorem in hyperbolic plane (see \cite{Bea}), which is the tool for us to elaborate the structure of certain $\C$-Fuchsian subgroups in complex hyperbolic triangle lattices, then give the proof of Theorem \ref{thm:main}.
\begin{thm}\label{thm:tool}
Let $D$ be a polygon in the hyperbolic plane satisfying the following conditions and denote $D\cup \partial D$ by $\bar{D}$.\\
\emph{(i)} For each side $s$ of $D$, there is a side $s'$ and an element $g_s$ (of the isometries of the hyperbolic plane) such that $g_s(s)=s'$, we call each $g_s$ the side pairing transformation.\\
\emph{(ii)} $g_{s'}=g_s^{-1}$. Observe that if there is a side $s$, with $s'=s,$ then it implies that $g_s^2=Id.$ If this occurs, the relation $g_s^2=Id$ is called a reflection relation. Now let $G$ be the group generated by the $g_s's$.\\ 
\emph{(iii)} $g_s(D)\cap D=\varnothing.$\\
\emph{(iv)} For each vertex $x$ of $D$, there are vertices $x_0 (=x), x_1, \cdots, x_n$ of $D$ and elements $f_0 (=Id), f_1, \cdots, f_n$ of $G$ such that the sets $f_j(N_j)$ ($N_j=\{y\in \bar{D}: d(y, x_j)<\epsilon \}$) are non-overlapping sets whose union is $B(x, \epsilon)$ (the ball centered at 
$x$ with radius $\epsilon$) and such that each $f_{j+1}$ is of the form $f_jg_s$ for some $s$ ($j=1, \cdots, n; f_{n+1}=Id$).\\
\emph{(v)} The $\epsilon$ in the above condition can be chosen independently of $x$ in $\bar{D}$.\\
Then the group $G$ generated by the side pairing transformations is discrete, and $D$ is a fundamental polygon for $G.$

\end{thm}
%In what follows we study the $\C$-Fuchsian subgroups in some non-arithmetic complex hyperbolic lattices, each of which is type of sporadic triangle group or Thompson group. 
Before we give the proof of Theorem \ref{thm:main}, we state the Cosine Rule for a hyperbolic triangle; this will be our tool of checking the local tiling condition near each vertex below. 
\begin{prop}
Suppose that a hyperbolic triangle have sides $a,$ $b$ and $c$ and opposite angles $\alpha,$ $\beta,$ and  $\gamma.$ Then the following formula holds
\begin{equation}\label{Cosine}
\cosh c=\cosh a\cosh b- \sinh a\sinh b \cos \gamma.
\end{equation}
\end{prop}
For the details, see  Section 7.12 in \cite{Bea}. 

%When we check the fundamental domain for each $\C$-Fuchsian subgroup, we will repeatedly use the following simple fact.
%\begin{lem}
%Assume that $g\in\PU(2,1)$ and two distinct points $z_1, z_2\in\overline{\HC}.$ If $z_1, z_2$ lie in a complex geodesic $L$ and  $g(z_1), g(z_2)$ also lie in complex geodesic $L$, then $g(L)=L.$
%\end{lem}
\begin{proof}[Proof of Theorem \ref{thm:main}]

Suppose that ${\bf n}_i$ is the polar vector of $R_i\,(i=1, 2, 3)$ and $u=e^{i\phi/3}=e^{2ip/3}$. By the trace formula of $\tr(R_1J)$ in \cite{Par3}, we may write $\tau$ as
\[
\tau=\tr(R_1J)=(u^2-\bar{u}){\bf\frac{\langle  n_{j+1}, n_{j}\rangle}{\|n_{j+1}\| \|n_j\|}}.
\]
We normalise $\n_i$ so that $\langle \n_i, \n_i\rangle=2-u^3-\bar{u}^3$. Then one can get that $\langle \n_{i+1}, \n_i\rangle=(\bar{u}^2-u)\tau.$  We now choose the polar vectors $\bf{n_i}$ of the complex geodesics $R_i$ $(i=1, 2, 3)$ to be the normal basis of $\C^3$, i.e., 
\begin{equation}
{\bf{n}}_1=\left[\begin{matrix}
  1\\
  0\\
  0
\end{matrix}\right],\quad
{\bf{n}}_2=\left[\begin{matrix}
  0\\
  1\\
  0
\end{matrix}\right],\quad
{\bf{n}}_3=\left[\begin{matrix}
  0\\
  0\\
  1
\end{matrix}\right].
\end{equation}
Therefore the corresponding matrix representation of complex hyperbolic isometry $J$ and the Hermitian form $H$ are given respectively by 
\begin{equation}
J=\left(\begin{matrix}
 0\quad&0\quad&1\\
 1\quad&0\quad&0\\
0\quad&1\quad&0
\end{matrix}\right),
\qquad
\H=\left(\begin{matrix}
  \alpha &\quad \beta &\quad \bar{\beta}\\
  \bar{\beta}&\quad \alpha&\quad \beta\\
  \beta&\quad \bar{\beta}&\quad \alpha
\end{matrix}\right).
\end{equation}
where $\alpha=2-u^3-\bar{u}^3,\, \beta=(\bar{u}^2-u)\tau.$ Then the Hermitian form is of signature $(2, 1)$ if and only if
$${\rm{det}(H)}=\alpha^3+2\Re (\beta^3)-3\alpha|\beta|^2<0.$$
All the lattices we will consider below satisfy the above inequality.
We can get the matrix representation of $R_1$ in $\rm{SU}(2, 1)$ by the formula (\ref{ref:l})
\begin{equation}
R_1=\left(\begin{matrix}
  u^2 &\quad \tau &\quad -u\bar{\tau}\\
  0&\quad \bar{u}&\quad 0\\
  0&\quad 0&\quad \bar{u}
\end{matrix}\right).
\end{equation}
  Correspondingly, one can get the matrix forms of $R_2, R_3$ by the relations
$$R_2=JR_1J^{-1},\, R_3=JR_2J^{-1}.$$

Let $\bf{v}_i$ ($i=1, 2, 3$) denotes a lift of the three vertices of the triangle, i.e., ${\bf{v}}_i=\n_{i+1}\boxtimes\n_{i+2}$. A direct computation
yields 
$$
{\bf{v}}_1=\left[\begin{matrix}
 \alpha^2-|\beta|^2\\
 \beta^2-\alpha\bar{\beta}\\
 -\alpha\beta+\bar{\beta}^2
\end{matrix}\right],\quad
{\bf{v}}_2=\left[\begin{matrix}
   -\alpha\beta+\bar{\beta}^2\\
 \alpha^2-|\beta|^2\\
 \beta^2-\alpha\bar{\beta}
\end{matrix}\right],\quad
{\bf{v}}_3=\left[\begin{matrix}
  \beta^2-\alpha\bar{\beta}\\
  -\alpha\beta+\bar{\beta}^2\\
 \alpha^2-|\beta|^2
\end{matrix}\right].
$$
%and all of them lie in the complex hyperbolic space.

%We denote $tr(R_1J)$ by $\tau$, i.e. $\tau=a\bar{\omega}(u^2-\bar{u}).$ 
In what follows, we investigate the subgroups of triangle lattices $\mathcal{S}(p, \tau)$ (see Table 1), which fix complex geodesic $L_1$ by considering fundamental domains of the triangle lattices one by one. We refer to Appendix in \cite{DPP2} for the details of the explicit presentations of triangle lattices $\mathcal{S}(p, \tau)$ and the combinatorial invariant shells.
We wish to emphasize that the invariant shells are the side representatives of the fundamental domains for the complex hyperbolic lattices.
%we will investigate all the elements, which fix complex geodesic $L_1$. There are actually infinity elements for each sporadic triangle group, but we just need to consider finite elements by checking the facets of their fundamental domains (please refer the details to \cite{DPP2, DPP}).

%the representation of a Fuchsian subgroup fixing $L_{1}$ one by one case. For each representation of complex hyperbolic lattice in sporadic triangle group or non-equilateral triangle group, please refer the details to \cite{DPP2}.

%%%%%%%%%%%%%%%%%%%%%%%%%%%%%%%%%%%%%%%%\UTF{00B5}\UTF{00DA}\UTF{00D2}\UTF{00BB}\UTF{00D6}\UTF{00D6}\UTF{00C7}\UTF{00E9}\UTF{00D0}\UTF{00CE}  \sigma1
\vskip1\baselineskip
\noindent\textbf{(\expandafter{\romannumeral1})} $\tau=-1+i\sqrt{2}.$

The triangle lattice $\Gamma$ is generated by $R_1, R_2, R_3, J$, explicitly
\begin{equation}\label{rep-s1}
\begin{aligned}
\langle R_1, R_2, R_3, J: \, &R_1^p, J^3, (R_1J)^8, R_3=JR_2J^{-1}=J^{-1}R_1J, (R_1R_2)^{|\frac{3p}{p-3}}|,
br_3(R_1, R_2R_3R_2R_3^{-1}R_2^{-1}),\\
& br_6(R_1, R_2), br_4(R_1, R_2R_3R_2^{-1}),(R_1R_2R_3R_2^{-1})^{|\frac{4p}{p-4}|}, 
br_3(R_1, R_3^{-1}R_2^{-1}R_3R_2R_3)\rangle
\end{aligned}
\end{equation}
Throughout the paper, relations involving infinite exponents shall be removed from the presentation.
In the form of a list of side representatives of  $\Gamma$'s fundamental domain, the rough structure of the invariant shells is given by 
\begin{equation}\label{shell1}
[6]\,1; 2, 3; [4] \,2; 1, 23\bar{2}; [3]\, 23\bar{2};1, 232\bar{3}\bar{2}; [3]\, 232\bar{3}\bar{2}; 1, \bar{3}\bar{2}323,
\end{equation}
where $[k] \,a; b, c$ denotes a $k$-gon pyramid with base $L_a$ (which is fixed by element $a$). In Figure 1, we give a rough picture of $[6]\,1; 2, 3$.
%Then one easily get that $2(x_0)=x_2$
\begin{figure}[h]\label{pya}
\includegraphics[height=1.8in]{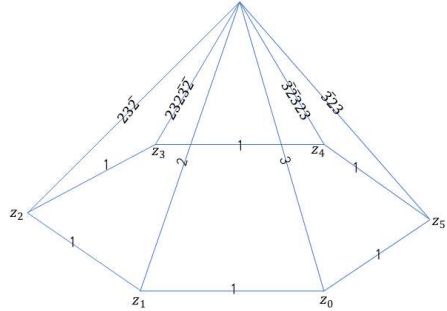}
\caption{Pyramid corresponding to $[6]\,1; 2, 3$ with the base $L_1$ fixed by the complex reflection $R_1$. Note that 
$\bar{3}\bar{2}\bar{3}2323=232\bar{3}\bar{2}$ is a consequence of the braid relation $br_6(R_1, R_2).$}
\end{figure}

Here each vertex ${\bf z}_i$ is the intersection point of the lateral edge with the base edge $L_1$, therefore usually the formula of the vertices can be written in this form:  ${\bf z}_1=\n_1\boxtimes\n_2,$ ${\bf z}_2=\n_1\boxtimes R_2(\n_3)$ and so on. However, one should note that the form of each vertex of such a pyramid depends on $p$; for example, the vertex ${\bf z}_2$ (the intersection point of  $R_2(L_3)$ with $L_1$) will be slightly changed when $p=6.$ One can check that $\n_1\boxtimes R_2(\n_3)$ is a positive vector 
%not lying in $\overline{\HC}.$ 
which is also a polar vector of the common perpendicular complex geodesic $L_{123\bar{2}}$ to $L_1$ and $R_2(L_3).$
Actually, the point ${\bf z}_2$ will be $\n_1\boxtimes(\n_1\boxtimes 2(\n_3)).$ 

%CHOOSE the shells which has the base $L_1$.
We take elements of the triangle lattice such that each of the four shells in (\ref{shell1}) to a pyramid with the same base $L_1$. Note that all of them indeed exist among their tessellation to the complex hyperbolic plane.  We leave $[6]\,1; 2, 3$ invariant and do minor surgeries to the other three pyramids such that each of them has  base $L_1$. Firstly, we consider the action of  the element $J^{-1}$  on the pyramid $[4] \,2; 1, 23\bar{2}$ with  base $L_2.$ Then one can get a new pyramid with base $J^{-1}(L_2)$ fixed by $J^{-1}2J$. The new pyramid is identified with $[4] \,J^{-1}2J; J^{-1}1J, J^{-1}23\bar{2}J$ which can be written as $[4] \, 1; 3, 12\bar{1}$ due to $R_{j+1}=JR_jJ^{-1}.$ 
Similarly, we deform the other two pyramids in (\ref{shell1}) to be with the same base $L_1:$ 
\begin{align*}
&[4] \,2; 1, 23\bar{2}\stackrel{J^{-1}}{\longrightarrow}[4] \, 1; 3, 12\bar{1}\\
&[3]\, 23\bar{2};1, 232\bar{3}\bar{2}\stackrel{J}{\longrightarrow}[3]\,31\bar{3}; 2, 313\bar{1}\bar{3}
\stackrel{\bar{3}}{\longrightarrow}[3]\, 1; \bar{3}23, 13\bar{1}\\
&[3]\, 232\bar{3}\bar{2}; 1, \bar{3}\bar{2}323\stackrel{J^{-1}}{\longrightarrow}[3]\,121\bar{2}\bar{1};3, \bar{2}\bar{1}212
\stackrel{\bar{2}\bar{1}}{\longrightarrow}[3]\,1;\bar{2}\bar{1}312, 12\bar{1}\\
\end{align*}
Along the same base $L_1$, we paste the four pyramids
\begin{equation}\label{sig1}
[6]\,1; 2, 3; [4] \, 1; 3, 12\bar{1}; [3]\, 1; \bar{3}23, 13\bar{1}; [3]\,1;\bar{2}\bar{1}312, 12\bar{1},
\end{equation}
%which have hexagonal base, square base, triangular base and triangular base  respectively. %Because we would like to find the Fuchsian group fixing complex geodesic $L_1,$Now we only need to pay attention to the bases in $L_1,$ which contains  
and mainly focus on the obtained decagon which lies in the closure of complex geodesic $L_1$ (homeomorphic to $\overline{{{\textbf{H}_{\mathbb{C}}^1}}}$). Its vertices are the intersection points of complex geodesic $L_1$ with other complex geodesics. Generally, the decagon $F$ (see Figure 2) has vertices $x_j={\bf{P}}(\n_1\boxtimes \a_j)$ ($j=0,1,\cdots, 9$) where 
\begin{figure}[htp]\label{fig:F.D1}
\includegraphics[height=2.5in]{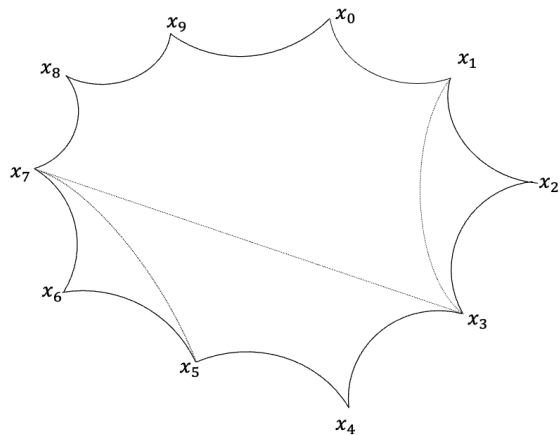}
\caption{The decagon $F$}
\end{figure}
\begin{alignat*}{4}
&\a_0=\bar{3}\bar{2}\n_3,\,
\a_1=\bar{3}\n_2, \quad
\a_2=\bar{3}231\n_3, \,
\a_3=\n_3,\, 
 \a_4=31\n_2,\,\\
 &\a_5=\bar{2}\bar{1}\n_3,\,
 \a_6= \bar{2}\bar{1}3121\n_2, \,
\a_7=\n_2,\, 
\a_8= 2\n_3, \,
 \a_9=23\n_2.\\
\end{alignat*}

Here $x_0$ just denotes the intersection point of $L_1$ with $R_3^{-1}R_2^{-1}(L_3)$ (the point fixed by the complex reflection  $\bar{3}\bar{2}323$), i.e., $x_0$ is the fixed point of $1\bar{3}\bar{2}323.$ Just as we stated previously, one should note that the formulas of the vertices above depend on $p.$ For example, the formula for the point $x_7$ (the intersection point of $L_1$ with $L_2$) will be ${\bf{P}}(\n_1\boxtimes(\n_1\boxtimes\n_2))$ when $p=4,$ which lies in the complex geodesic $L_1$ and is fixed by $12.$

From the combinatorics of the four pyramids (\ref{sig1}), we know that the decagon composes of  a hexagon $P_1$ with vertices $x_0,\, x_1,\, x_3,\, x_7,\, x_8,\, x_9$,
a quadrilateral $P_2$ with vertices $x_3,\, x_4,\, x_5,\, x_7,$ a triangle  $P_3$ with vertices $x_1,\, x_2,\, x_3$ and 
a triangle  $P_4$ with vertices $x_5,\, x_6,\, x_7$, i.e., it comprises ten sides $l_1,l_2,\cdots l_{10},$ where
\begin{equation}\label{sig1-s}
l_{i}={\bf{P}}({\rm{Span}}_\C\{x_{i-1}, x_i\})\cap L_1.
\end{equation}
% where all of them are pairwise distinct and obviously lie in the complex geodesic $L_1$.
 
In order to find the explicit structure of the Fuchsian group stabilising $L_1$, we start from the side pairing transformations for the decagon. One can separately  consider the transformations which convert vertices from the construction of above four polygons. Note that the element $\bar{3}23$ transfers the complex geodesic $L_{\bar{3}\bar{2}323}$ fixed by $\bar{3}\bar{2}323$ to the complex geodesic $L_{3}$ fixed by $R_3$ when focusing on the hexagon $P_1.$ Also, the element $\bar{3}23$ transfers the complex geodesic $L_{3}$ to the complex geodesic $L_{\bar{3}2313\bar{1}\bar{3}\bar{2}3}$ when focusing on the triangle $P_3.$ Then we consider the element $(1\bar{3}23)^2$ which fixes the vertex $x_1$ and obtain that $(1\bar{3}23)^2(l_1)=l_2.$ In the same manner, one can get that $(13)^3$ fixes $x_3$ and maps $l_3$ to $l_4$; $ (12)^3$ fixes $x_7$ and maps $l_7$ to $l_8.$  Now, let $g_1=(1\bar{3}23)^2,\, g_2=(13)^3,\, g_3=(12)^3.$ One can know that
$$g_1(l_1)=l_2,\, g_2(l_3)=l_4,\, g_3(l_7)=l_8.$$

The difficulty here is pairing the remaining four sides $l_5,\, l_6,\, l_9,\, l_{10}.$ We have that $(12)^3(x_6)$ $=x_8,$ and  pay attention to the stabiliser $123\bar{2}$ of $x_8.$ Denote $(123\bar{2})^2(12)^3$ by $ g_4,$ one can immediately get that $g_4(x_6)=x_8.$ We claim that $g_4(x_5)=x_9.$
Note that $(123)^3=1J,$ because the order of $1J$ is 8 and $(1J)^3=123.$ Using ${\textrm {br}}_6(1, 2),$ we have
\begin{align*}
&\bar{1}^2(123\bar{2})^2(12)^3(\a_5)\\
= &\bar{1}^2(123\bar{2})^2(212121)\bar{1}(\bar{2}\bar{1}\,\n_3)\\
=&\bar{1}23\bar{2}123123\,\n_3\\
=&\bar{1}23\bar{2}\bar{3}\bar{2}J\,\n_3\\
=&\bar{1}23\bar{2}\bar{3}\bar{2}\,\n_1.
\end{align*}
Due to ${\rm{br}}_3(1,232\bar{3}\bar{2}),$ we see at once that $\bar{1}23\bar{2}\bar{3}\bar{2}1232\bar{3}\bar{2}1=232\bar{3}\bar{2},$ which means
$\bar{1}23\bar{2}\bar{3}\bar{2}\,\n_1=23\,\n_2=\a_9.$ Therefore, we obtain that $g_4(l_6)=l_9.$ 
%A direct computation yields that $(123\bar{2})^2\circ (12)^3(x_6)=x_8$ and $(123\bar{2})^2\circ (12)^3(x_5)=x_9,$ i.e.,$$(123\bar{2})^2\circ (12)^3(l_6)=l_9.$$Similarly, one can get that $(1232\bar{3}\bar{2})^3\circ (123\bar{2})^2\circ (12)^3(l_5)=l_{10}.$

Let $g_5=(1232\bar{3}\bar{2})^3\circ g_4=(1232\bar{3}\bar{2})^3(123\bar{2})^2(12)^3.$ It follows immediately that $g_5(x_5)=x_9,$ because $g_4(x_5)=x_9$ and $1232\bar{3}\bar{2}$ fixes $x_9.$ We claim that $g_5(x_4)=x_0$ by checking  
$\bar{1}^3g_5(\bar{1}\a_4)=1\a_0.$ Considering the braid relations in the group presentation (\ref{rep-s1}), we have
\begin{align*}
&\bar{1}^3(1232\bar{3}\bar{2})^3(123\bar{2})^2(12)^3(\bar{1}\a_4)\\
=&\bar{1}^3(1232\bar{3}\bar{2})^3(23\bar{2}1)^2(21)^3(\bar{1}31\n_2)\\
=&\bar{1}232\bar{3}\bar{2}1(123)^2\bar{2}(12)^3(31\n_2)\\
=&\bar{1}232\bar{3}\bar{2}1(\bar{3}\bar{2}J)\bar{2}(21)^3(31\n_2)\\
%=&\bar{1}232\bar{3}\bar{2}1\bar{3}\bar{2}J\bar{2}(212121)31\n_2\\
=&\bar{1}232\bar{3}\bar{2}1\bar{3}\bar{2}J1212131\n_2\\
%=&\bar{1}232\bar{3}\bar{2}1\bar{3}\bar{2}(2323212J)\n_2\\
=&\bar{1}(232\bar{3}\bar{2}1232\bar{3}\bar{2})2312J\n_2\\
%=&\bar{1}(1232\bar{3}\bar{2}1)23123\n_3\\
=&232\bar{3}\bar{2}(123)^2\n_3\\
=&232\bar{3}\bar{2}\bar{3}\bar{2}J\n_3\\
%=&(232)\bar{3}\bar{2}\bar{3}\bar{2}\n_1\\
=&\bar{3}\bar{2}\bar{3}(23)^3(\bar{3}\bar{2})^2\n_1\\
=&\bar{3}\bar{2}\bar{3}23\n_1.
\end{align*}
It is easy to know that $\bar{3}\bar{2}\bar{3}23\n_1=1\bar{3}\bar{2}\n_3=1\a_0,$ because of $(\bar{3}\bar{2}\bar{3}23)1(\bar{3}\bar{2}323)=1\bar{3}\bar{2}323\bar{1}$ from ${\rm {br}}_3(1, \bar{3}\bar{2}323).$  Therefore $g_5(x_4)=x_0,$ i.e., $g_5(l_5)=l_{10}.$

We consider the $\mathbb{C}$-Fuchsian subgroup $\Gamma_0$ generated by $g_1, g_2, g_3, g_4, g_5,$
and only need to check the local tiling condition near every vertex of the decagon $F$ in Figure 2 to show that it is indeed the fundamental domain of the $\mathbb{C}$-Fuchsian subgroup $\Gamma_0$. It is sufficient to consider the three cycles: $\{x_0,\, x_2,\, x_4\},\,\{x_5,\, x_9\},\,\{x_6,\, x_8\}.$ We would like to take the vertex $x_8$ when $p=3$ for example.  It is easily seen that the order of the stabiliser $g_3\circ g_4^{-1}$ in $\Gamma_0$ of the vertex $x_8$ is 6 by considering its eigenvalues and eigenvectors. Because the invariant shells (\ref{shell1}) are the side representatives of the fundamental domain of the complex hyperbolic lattice $\Gamma,$ the domains $(g_3\circ g_4^{-1})^m(F)$ do not intersect with each other for $m=1,\, 2,\, 3,\, 4,\, 5.$  Let $\theta_1,$ $\theta_2$  be the internal angles of $x_6,\, x_8$ respectively. Using the Cosine Rule  (\ref{Cosine}) for the triangle with vertices $g_3(x_5),\, x_8,\, x_7$ lying in the complex geodesic $L_1$ (an embedded copy of ${\bf H}^1_\mathbb{C}$), one can get that $\theta_1+\theta_2$ is exactly $2\pi/6.$ We conclude, therefore, that the local tiling condition is satisfied for the vertex $x_8$ when $p=3.$ In the same way, one can check the local tiling condition of all vertices for $p=3,\, 4,\, 6$ by considering the relation of the sum of angles at all elliptic vertices belonging to an elliptic cycle with the order of that cycle. In particular, by computing the the angle of the elliptic cycle, we could get the order of the three cycle transformations at $x_0,$ $x_5,$ $x_8$ respectively. Finally we obtain that  the decagon $F$ is the fundamental domain of  the $\mathbb{C}$-Fuchsian subgroup $\Gamma_0$ by Theorem \ref{thm:tool}.

From  the presentation (\ref{rep-s1}) of the triangle lattice. We see at once that both $g_2$ and $g_3$ is of order $|\frac{p}{p-3}|.$ Recalling the element $P^2=1J1J,$ we have
\begin{align*}
&P^21P^{-2}=1J1J1J^{-1}\bar{1}J^{-1}\bar{1}=1(23\bar{2})\bar{1},\\
&P^2(\bar{3}23)P^{-2}=1J1J\bar{3}23J^{-1}\bar{1}J^{-1}\bar{1}=1,
\end{align*}
which yields that the order of $g_1=(1\bar{3}23)$ is $|\frac{2p}{p-4}|.$
One can obtain the normal representations of the $\mathbb{C}$-Fuchsian subgroup $\Gamma_0$ for $p=3, 4, 6:$
\begin{itemize}
\item$p=3,$ $4:$
\[
\langle g_1, g_2, g_3, g_4, g_5:  g_1^{|\frac{2p}{p-4}|}, \, g_2^{|\frac{p}{p-3}|},\, g_3^{|\frac{p}{p-3}|},
 (g_5\circ g_2\circ g_1)^{|\frac{2p}{p-2}|}, \, (g_5\circ g_4^{-1})^{|\frac{2p}{p-4}|}, \, (g_4\circ g_3^{-1})^{|\frac{2p}{p-4}|}\rangle,
\]
\\
%\begin{figure}
 % \includegraphics[width=10cm]{tu}
  %\caption{rough fundamental domain}
%\end{figure}
\item $p=6:$
\[
\langle g_1, g_2, g_3, g_4, g_5:  g_1^6, \, g_2^2, \,  g_3^2, \, (g_4\circ g_3^{-1})^6\rangle.
\]
where $g_5\circ g_2\circ g_1,$ $(g_5\circ g_4^{-1})$ are ellipto-parabolic elements. In this case $x_0,$ $x_5$ lie in the boundary of the disk.
%are all complex reflections in complex geodesics, $E$ is a loxodromic element.
\end{itemize}
\vskip1\baselineskip
\noindent\textbf{(\expandafter{\romannumeral2})} $\tau=-\frac{1+i\sqrt{7}}{2}.$ 

The triangle lattice $\Gamma$ is generated by $R_1, R_2, R_3, J$, explicitly
\begin{equation}\label{rep:s4}
\begin{aligned}
\langle R_1, R_2, R_3, J: \, &R_1^p, J^3, (R_1J)^7, R_3=JR_2J^{-1}=J^{-1}R_1J, (R_1R_2)^{|\frac{4p}{p-4}|}, \\
& br_4(R_1, R_2), (R_1R_2R_3R_2^{-1})^{|\frac{6p}{p-6}|}, br_3(R_1, R_2R_3R_2^{-1}) 
\rangle
\end{aligned}
\end{equation}
The rough structure of the invariant shell is given by
$$[4]\, 1;\, 2,\, 3; \quad [3]\, 2;\, 1,\, 23\bar{2}.$$
The element $J^{-1}$ maps the shell $[3]\, 2;\, 1,\, 23\bar{2}$ to $[3]\, 1;\, 3,\, 12\bar{1}.$ Pasting the two shells
\begin{equation}\label{sig4}
[4]\, 1;\, 2,\, 3; \quad [3]\, 1;\, 3,\, 12\bar{1}
\end{equation}
 along the bases in $L_1,$ we get a pentagon (see Figure 3) with vertices $x_j={\bf{P}}(\n_1\boxtimes \a_j)$ where
\begin{align*}
&\a_1= 2\n_3, \quad \a_2= \bar{3}\n_2,\quad \a_3=\n_3,\quad \a_4= 31\n_2, \quad \a_5=\n_2.
\end{align*}
The pentagon composes of a tetragon $P_1$ with vertices $x_1, x_2, x_3, x_5$ and a triangle $P_2$ with vertices $x_3, x_4, x_5.$ 
\begin{figure}[htp]\label{F.D.2}
\includegraphics[height=1.8in]{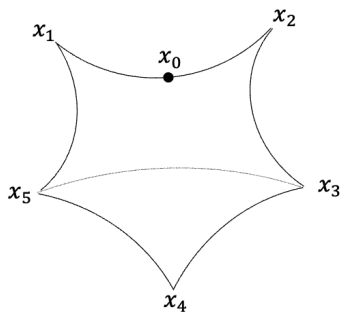}
\caption{The polygon $F$  }
\end{figure}
%where $x_1, x_2, x_3$ are pairwise distinct and lie in $\partial\HC.$
%Recalling the singular points existed in \cite{DPP2}, 
We restrict to the singular point $x_0$ fixed by $23\bar{2}P^2$  (where $P=R_1J$). A direct computation yields that $x_0$ lies in the  geodesic spanned by $x_1, x_2$ and $23\bar{2}P^2$ maps $x_1$ to $x_2.$ Indeed, using $(23\bar{2}121)2(\bar{1}\bar{2}\bar{1}2\bar{3}\bar{2})=(23)2(\bar{3}\bar{2}),$ we have
\begin{align*}
23\bar{2}P^2&(\a_1)=23\bar{2}1J1J(2\n_3)=23\bar{2}121\n_2=23\n_2=\bar{3}\n_2=\a_2.
\end{align*}
%because $(23\bar{2}121)2(\bar{1}\bar{2}\bar{1}2\bar{3}\bar{2})=(23)2(\bar{3}\bar{2}).$
Thence we consider a hexagon comprising the following sides
\begin{equation}\label{sig4-s}
l_{i}=\left\{
             \begin{array}{ll}
             {\bf{P}}({\rm{Span}}_\C\{x_{i}, x_0\})\cap L_1,  &\quad i=1, 2.  \\
             \\
             {\bf{P}}({\rm{Span}}_\C\{x_{i-1}, x_i\})\cap L_1, &\quad i=3, 4, 5.  
             \end{array}
\right.
\end{equation}
By the composition of this hexagon, one could check that the second power of the element $12$ maps $x_4$ to $x_1$. Then we get that $(12)^2(l_5)=l_6$ because of  $(12)^2(x_4)=x_1$ and $(12)^2(x_5)=x_5.$ 
Similarly, it is easy to check that $(13)^2(l_3)=l_4.$ Furthermore, it follows from $23\bar{2}P^2(x_0)=x_0,\, 23\bar{2}P^2(x_1)=x_2$ that $23\bar{2}P^2(l_1)=l_2.$   Now we have paired the sides of the pentagon drawn above. For each vertex of the cycle $\{ x_1,\, x_2,\, x_4\},$ one can verify that it satisfies  local tiling condition by considering the sum of angle of all vertices and the order of its stabiliser.
%\tau=-\frac{1+i\sqrt{7}}{2}  (12)^2,(13)^2,23\bar{2}P^2

By Theorem \ref{thm:tool}, we obtain the fundamental domain in $L_1$ (which is the hexagon $F$ in Figure 3) of the $\mathbb{C}$-Fuchsian subgroup generated by $g_1, g_2, g_3,$ where
$$g_1=(12)^2,\quad g_2=23\bar{2}P^2, \quad g_3=(13)^2.$$
 %, which can be checked by directly checking the  {\textcolor{red}{eigenvalues and eigenvectors traces}}.\\
%By the relation shown in the triangle lattice, we could write the representation of a Fuchsian subgroup fixing $L_{1}$ in what follows, also obtain the corresponding fundamental domain.
We are now in a position to show the presentation of this $\mathbb{C}$-Fuchsian subgroup for all values of $p.$
Because the order of $1J$ is 7, it follows from $P=1J$ and $(1J)^3=123$ that $P^2=P^{-5}=P^{-2}\bar{3}\bar{2}\bar{1}=J^{-1}\bar{1}J^{-1}\bar{1}\bar{3}\bar{2}\bar{1}.$ We have that
\[
g_2^2=(23\bar{2}1J1J)(23\bar{2}J^{-1}\bar{1}J^{-1}\bar{1}\bar{3}\bar{2}\bar{1})
%=23\bar{2}1J1J^{-1}J^{-1}23\bar{2}JJ\bar{1}J^{-1}\bar{1}\bar{3}\bar{2}\bar{1}
=23\bar{2}1212\bar{1}\bar{2}\bar{1}\bar{3}\bar{2}\bar{1}=\bar{1},
\]
which shows that the order of $g_2$ is $2p.$ From the braid relation ${\rm {br}_4}(1, 2)$, one can easily see that ${\rm {br}_4}(3, 1).$ Then we get that the order of $g_3$ is $|\frac{2p}{p-4}|$ because the elements $12,$ $23,$ and $31$ have the same order $|\frac{4p}{p-4}|.$ What is left is to consider the order of the elliptic cycle. 
Note that $P^2=(1J)^{-12}=(123)^{-4}=(\bar{3}\bar{2}\bar{1})^4.$ Using the relations in the presentation (\ref{rep:s4}), we have
\begin{align*}
g_1\circ g_3\circ g_2
=& 1^2(12)^3\bar{1}^2(13)^2(23\bar{2}P^2)\\
=&1^3212\bar{1}31(323\bar{2}\bar{3}\bar{2})\bar{1}(\bar{3}\bar{2}\bar{1})^3\\
=&1^3212\bar{1}31\bar{2}(3\bar{1}\bar{3})\bar{2}\bar{1}(\bar{3}1\bar{1}\bar{2}\bar{1}\bar{3}\bar{2}\bar{1})\\
=&1^32(12\bar{1}31\bar{2}\bar{1}\bar{3})(\bar{1}31\bar{2}\bar{1}\bar{3}1)\bar{1}\bar{2}\bar{1}\bar{3}\bar{2}\bar{1}\\
%=&1^32(\bar{3}12\bar{1})(\bar{2}\bar{1}\bar{3}12)\bar{1}\bar{2}\bar{1}\bar{3}\bar{2}\bar{1}\\
=&1^32\bar{3}(12\bar{1}\bar{2}\bar{1})\bar{3}(12\bar{1}\bar{2}\bar{1})\bar{3}\bar{2}\bar{1}\\
=&1^3(2\bar{3}\bar{2}\bar{1})^3
\end{align*}
which indicates the order of $g_1\circ g_3\circ g_2$ is $|\frac{2p}{p-6}|,$ since the order of $123\bar{2}$ is 
$|\frac{6p}{p-6}|$ and $1^3(2\bar{3}\bar{2}\bar{1})^3=(2\bar{3}\bar{2}\bar{1})^{3}1^3.$
%We elaborate the structure of the Fuchsian groups for $p=3, 4, 5, 6, 8, 12$ in what follows.
We get the following presentation of the $\mathbb{C}$-Fuchsian subgroup for $\tau=-\frac{1+i\sqrt{7}}{2}:$
\[
\langle g_1,\, g_2,\, g_3\,: g_1^{|\frac{2p}{p-4}|},\, g_2^{2p},\, g_3^{|\frac{2p}{p-4}|},\,
( g_1\circ g_3\circ g_2)^{|\frac{2p}{p-6}|}  \rangle.
\]

\vskip1\baselineskip
\noindent\textbf{(\expandafter{\romannumeral3})} $\tau=\frac{1+\sqrt{5}}{2}.$

The triangle lattice $\Gamma$ is generated by $R_1, R_2, R_3, J$, explicitly
\begin{align*}
\langle R_1, R_2, R_3, J: \, &R_1^p, J^3, (R_1J)^5, R_3=JR_2J^{-1}=J^{-1}R_1J,\, {\rm br}_5(R_1, R_2),\\ &(R_1R_2)^{|\frac{10p}{3p-10}|},\, {\rm br}_3(R_1, R_2R_3R_2^{-1}),(R_1R_2R_3R_2^{-1})^{|\frac{6p}{p-6}|}\rangle
\end{align*}
We consider the combinatorics of the fundamental domain for this triangle lattice which comprises the following two invariant shells
$$[5]\, 1;\, 2,\, 3;\quad [3]\, 2;\, 1,\, 23\bar{2}.$$

Following the process of the previous cases, we firstly map the pyramid $[3]\, 2;\, 1,\, 23\bar{2}$ to $[3]\, 1;\, 3,\, 12\bar{1}$  by the action of the element $J^{-1}.$ We paste the two pyramids 
\begin{equation}\label{sig10}
[5]\, 1;\, 2,\, 3; \quad [3]\, 1;\, 3,\, 12\bar{1}
\end{equation}
along the bases in $L_1.$ Then we get a hexagon in Figure 4 with the vertices $x_j={\bf P}({\bf{n_1}}\boxtimes \a_j)$ ($j=0,1,\cdots5$), where
\begin{align*}
\a_0=\bar{3}\bar{2}{\bf{n_3}},\quad \a_1= \bar{3}{\bf{n_2}},\quad \a_2= \n_3,
\quad\a_3=31{\bf{n_2}},\quad \a_4=\n_2,\quad \a_5=2{\bf{n_3}}.
\end{align*}
%where all of them are pairwise distinct and lie in $\HC.$ 
The hexagon composes of  a pentagon $P_1$ with vertices $x_0, x_1, x_2, x_4, x_5$ and a triangle $P_2$ with vertices $x_2, x_3, x_4.$ We write the sides of the hexagon $F_2$ as follows
\begin{equation}\label{sig10-s}
l_{i}={\bf{P}}({\rm{Span}}_\C\{x_{i-1}, x_i\})\cap L_1,\quad i=1,2, \cdots, 6.
\end{equation}
Let $g_1=1\bar{3}\bar{2}323,$ $g_2=1312\bar{1}\bar{3}$.
\begin{figure}[htp]\label{F.D.3}
\includegraphics[height=1.9in]{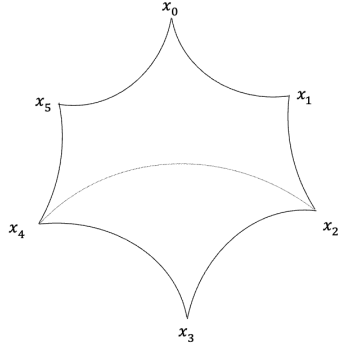}
\caption{The hexagon $F_2$ }
\end{figure}
A trivial verification shows that 
$$g_1(l_6)=l_1,\quad g_2(l_3)=l_4,$$
since $g_1(x_0)=x_0,\, g_1(x_5)=x_1, \, g_2(x_3)=x_3,\, g_2(x_2)=x_4.$ 
%by considering the explicit matrix representation of  $R_1, R_2, R_3.$ 
About the sides $l_5$ and $l_2$, we firstly apply the element $g_1$ to map $l_5$ to a geodesic $l$, one of whose endpoint is $x_1$ because of $g_1(x_5)=x_1$. Due to the construction of the fundamental domain for the triangle lattice $\Gamma$ and $1\bar{3}23$ fixes $x_1$, a direct calculation yields that $(1\bar{3}23)^3\cdot g_1$ maps $x_5, x_4$ to $x_2, x_1$ respectively.  Therefore, we know that 
$g_3(l_5)=l_2,$  where $g_3=(1\bar{3}23)^3\circ g_1.$  One could finally check the local tiling condition for the vertices of the two cycles: 
$\{ x_1,\, x_5\},$ $\{ x_2,\, x_4\}$ which follows by the same method as in the first case 
$\tau=-1+i\sqrt{2}.$
%$\tau=\frac{1+\sqrt{5}}{2} 1\bar{3}\bar{2}323, 1312\bar{1}\bar{3}, (1\bar{3}23)^3\cdot 1\bar{3}\bar{2}323$

We claim that for any holomorphic isometry $g$ fixing the complex geodesic $L_1,$ $g$ commutes with $R_1.$ Indeed, it follows $g(\n_1)=\n_1$ that $gR_1g^{-1}=R_1.$ i.e., $gR_1=R_1g.$ Similarly, $g$ also commutes with $R_1^{-1}.$ Note that $1J=(1J)^6=123123$ implies that $J=23123=31231=12312,$ also 
$(1J)^2=(1J)^{-3}=\bar{3}\bar{2}\bar{1}.$ Now we check the order of elliptic cycles at $x_2$ and $x_5$:
\begin{align*}
g_3\circ g_2&=(1\bar{3}23)^3(1\bar{3}\bar{2}323)(1312\bar{1}\bar{3})\\
&=1^3(1\bar{3}231\bar{3}231\bar{3}23)(\bar{3}\bar{2}323)(312\bar{1}\bar{3}\bar{1})\\
&=1^5\bar{1}\bar{3}\bar{1}\bar{3}(31231)\bar{3}(23123)(31231)\bar{1}\bar{3}\bar{1}\bar{3}\bar{1}\\
&=1^5(\bar{1}\bar{3})^2J\bar{3}J^2(\bar{3}\bar{1}\bar{3}\bar{1}\bar{3})\\
&=1^5(\bar{1}\bar{3})^5,\\
g_1^{-1}\circ g_3&=(1\bar{3}\bar{2}323)^{-1}(1\bar{3}23)^3(1\bar{3}\bar{2}323)\\
&=1^2(\bar{3}\bar{2}\bar{3}23)\bar{1}\bar{3}231\bar{3}23123\\
&=1^223\bar{2}(\bar{3}\bar{2}\bar{1})\bar{3}231\bar{3}(23123)\\
&=1^223\bar{2}1J1J^{-1}J^{-1}\bar{3}231\bar{3}J\\
&=1(123\bar{2}1)123\bar{2}\\
&=(123\bar{2})^3,
\end{align*}
which yield that the order of $g_3\circ g_2$ is $|\frac{2p}{3p-10}|,$ and the order of $g_1^{-1}\circ g_3$ is $|\frac{2p}{p-6}|.$
By Theorem \ref{thm:tool}, we obtain the fundamental domain in $L_1$ (which is a hexagon) of the $\mathbb{C}$-Fuchsian group generated by $g_1, g_2, g_3.$ 
We list the explicit presentation for all values of $p=3, 4, 5, 10:$
$$\langle g_1, \, g_2, \, g_3: g_1^p, \, g_2^p, \, (g_3\circ g_2)^{|\frac{2p}{3p-10}|}, \, (g_3^{-1}\circ g_1)^{|\frac{2p}{p-6}|} \rangle.$$
Now the proof is complete.
\end{proof}

%\emph{Remark 1.} We should note that the expression of the vertices for each fundamental domain existed above may need to be slightly changed necessarily, such as the case for $\tau=\sigma_1=-1+i\sqrt{2}$ and $p=4$. The complex geodesic $L_1$ will not intersect the complex geodesic $L_2$ in $\overline{\HC},$ then the  point $\n_1\boxtimes\n_2$ is the polar vector of a complex geodesic $L$ orthogonal to $L_1$ and $L_2$. In such case, we shall consider $f_c=(\n_1\boxtimes\n_2)\boxtimes\n_1$ which lies in the complex geodesic $L_1$ and is also fixed by $(R_1R_2)^3$ due to direct calculation.

% and is also fixed by $R_1R_2.$
%\emph{Remark 2.}
%Throughout this paper, we investigate the $\C$-Fuchsian subgroups fixing a complex geodesic. One should note that there exist conjugacy classes given by any element in the complex hyperbolic lattice. For example, we also could get a $\C$-Fuchsian subgroups  in $\mathcal{S}(3, \sigma_1)$ stabilising the complex geodesic $L_{21\bar{2}}$:
%$$\big\langle   2\circ(1\bar{3}23)^2\circ\bar{2}, 2\circ(13)^3\circ\bar{2}, 2\circ(12)^3\circ\bar{2}, 2\circ(123\bar{2})^2\circ\bar{2}, 2\circ(1232\bar{3}\bar{2})^3\circ (123\bar{2})^2\circ (12)^3\circ\bar{2} \big\rangle.$$

\begin{rem}\label{rem:Thom}
By an almost similar method, one can also consider the structure of $\mathbb{C}$-Fuchsian subgroups in Thompson triangle groups for $S_2$ and $E_2$ in Table 2. Recall the matrix normalisation (\ref{Thom:mat}) for Thompson triangle groups in section 3.2.  We describe these two cases roughly in what follows. We stress that calculations can be done in the same manner with the proof above and the situation is improving significantly when one uses Mathematica for example.
\end{rem}
%\vskip1\baselineskip
%\noindent\textbf{(\expandafter{\romannumeral1})} 
\noindent\textbf{(\expandafter{\romannumeral1})} Thompson triangle group $S_2.$  
\begin{align*}
\langle R_1, R_2, R_3: &R_1^p, R_2^p, R_3^p, (R_1R_2R_3)^5, {\rm{br}}_3(R_1, R_3),
 {\rm{br}}_3(R_2, R_3),\\& {\rm{br}}_4(R_1, R_2),
 (R_1R_2)^{|\frac{4p}{p-4}|}, {\rm{br}}_5(R_1, R_2R_3R_2^{-1}), (R_1R_2R_3R_2^{-1})^{|\frac{10p}{3p-10}|}  \rangle
\end{align*}
It has the following pyramids of the side representatives of its fundamental domain
\[
[3] 1; 2, 3, \quad [5] \,2; 1, 23\bar{2},\quad [4]\, 3;1, 2,\quad [3]\, 23\bar{2}; 1, 3.
\]
We force them to having the same base $L_1:$
\begin{align*}
&[5] \,2; 1, 23\bar{2}\stackrel{23}{\longrightarrow}[5] \, 3; 231\bar{3}\bar{2}, 2\stackrel{\bar{3}\bar{1}}{\longrightarrow}[5] \, 1;\bar{3}\bar{1}231\bar{3}\bar{2}13, \bar{3}\bar{1}213\\
&[4]\, 3;1, 2\stackrel{\bar{3}\bar{1}}{\longrightarrow}[4]\,1; \bar{3}13, \bar{3}\bar{1}213\\
&[3]\, 23\bar{2}; 1, 3\stackrel{\bar{2}}{\longrightarrow}[3]\, 3;\bar{2}12;\bar{2}32\stackrel{\bar{3}\bar{1}}{\longrightarrow}
[3]\, 1; \bar{3}\bar{1}\bar{2}1213, \bar{3}\bar{1}\bar{2}3213
\end{align*}
%Along the same base $L_1$, we paste the four pyramids
%\begin{equation}\label{sig1}
%[3]\,1; 2, 3; [5] \, 1;\bar{3}\bar{1}231\bar{3}\bar{2}13, \bar{3}\bar{1}213;[4]\,1; \bar{3}13, \bar{3}\bar{1}213; [3]\, 1; \bar{3}\bar{1}\bar{2}1213, \bar{3}\bar{1}\bar{2}3213,
%\end{equation}
We pay attention to the nonagon $F$ with vertices $x_j={\bf{P}}(\n_1\boxtimes \a_j)$ ($j=1,\cdots, 9$) as follows: 
\begin{figure}[htp]\label{fig:F.D.S2}
\includegraphics[height=2.5in]{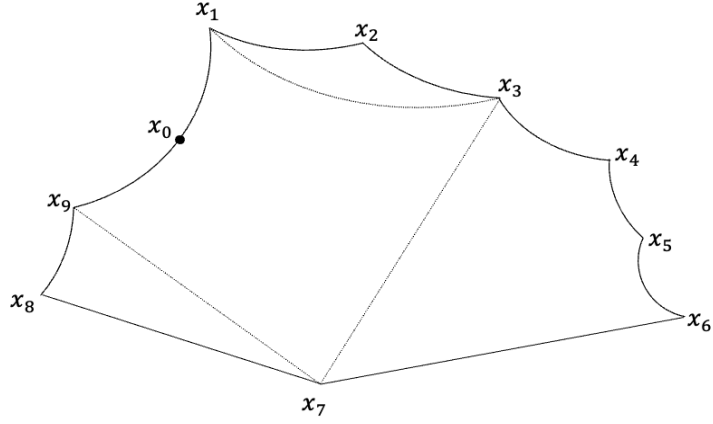}
\caption{The nonagon $F$}
\end{figure}
\begin{alignat*}{4}
&
\a_1=\bar{3}\bar{1}\bar{2}\n_1, \,
\a_2=\bar{3}\bar{1}\bar{2}1\n_3, \,
\a_3=\bar{3}\bar{1}\n_2,\, 
 \a_4=\bar{3}\bar{1}23\n_1,\,
 \a_5=\bar{3}\bar{1}2312\n_3,\,\\
 &\a_6= (\bar{3}\bar{1})^223123\bar{2}\n_1, \,
\a_7=\n_3,\, 
\a_8= \n_2, \,
 \a_9=2\n_3.
\end{alignat*}
Note that the singular point $x_0$ fixed by $\bar{3}\bar{1}Q^{-1}2Q^313$  (where $Q=R_1R_2R_3$),  lies in the  geodesic spanned by $x_1, x_9.$  A direct computation yields that $\bar{3}\bar{1}Q^{-1}2Q^313$ maps $x_1$ to $x_9.$ Thence we consider the decagon comprising the following sides
\begin{equation}\label{sig4-s}
l_{i}=\left\{
             \begin{array}{ll}
             {\bf{P}}({\rm{Span}}_\C\{x_{i-1}, x_i\})\cap L_1, &\quad i=1, 2, \cdots 9.\\
             \\
             {\bf{P}}({\rm{Span}}_\C\{x_{9}, x_0\})\cap L_1,  &\quad i=10.  
             \end{array}
\right.
\end{equation}
In the same manner with the proof of Theorem \ref{thm:main}, one can check that the following side pairing transformations
\[
g_1=(13)^3,\, g_2=(1\bar{3}\bar{1}213)^3,\, g_3=(1\bar{3}\bar{1}231\bar{3}\bar{2}13)^2\circ g_2,\,
g_4=(12)^2\circ g_1,\, g_5=\bar{3}\bar{1}Q^{-1}2Q^313
\]
satisfy
\[
g_1(l_7)=l_8, \,  g_2(l_3)=l_4,\,  g_3(l_2)=l_5, \,  g_4(l_6)=l_9, \, g_5(l_{10})=l_1.
\]
One can check the cycle transformation for $x_4$ satisfying 
$$(g_2\circ g_3^{-1})=1^3(2\bar{3}\bar{2}\bar{1})^2$$
and get the the presentation of $\C$-Fuchsian group  fixing $L_1$  in Thompson group $S_2:$
\[\langle g_1, g_2, g_3, g_4, g_5: g_1^{|\frac{2p}{p-6}|},\, g_2^{|\frac{2p}{p-6}|},\, g_5^{2p},\,  (g_2\circ g_3^{-1})^{|\frac{2p}{p-4}|},\, (g_1\circ g_4^{-1})^{|\frac{2p}{p-4}|}\rangle.
\]
%\begin{thm} \label{thm:thom1} 
%For Thompson triangle group $S_2$, there exist $\C$-Fuchsian subgroups stabilising $L_1$ which have the following structure when $p=3, 4$:\\
%\[\langle g_1, g_2, g_3, g_4, g_5: g_1^{|\frac{2p}{p-6}|},\, g_2^{|\frac{2p}{p-6}|},\, g_5^{2p},\,  (g_2\circ g_3^{-1})^{|\frac{2p}{p-4}|},\, (g_1\circ g_4^{-1})^{|\frac{2p}{p-4}|}\rangle.\]\end{thm}
\noindent\textbf{(\expandafter{\romannumeral2})} Thompson triangle group $E_2.$  
\begin{align*}
\langle &R_1, R_2, R_3: R_1^p, R_2^p, R_3^p, (R_1R_2R_3)^6, {\rm{br}}_3(R_2, R_3),
 {\rm{br}}_4(R_3, R_1), (R_1R_3)^{|\frac{4p}{p-4}|}, {\rm{br}}_4(R_1, R_2),\\
& (R_1R_2)^{|\frac{4p}{p-4}|}, {\rm{br}}_4(R_1, R_2R_3R_2^{-1}), (R_1R_2R_3R_2^{-1})^{|\frac{4p}{p-4}|},   {\rm{br}}_6(R_3, R_1R_2R_1^{-1}), (R_3R_1R_2R_1^{-1})^{|\frac{3p}{p-3}|}\rangle
\end{align*}
We restrict to the pyramids of the side representatives of its fundamental domain
\[
[3] 1; 2, 3, \quad [6] \,\bar{3}13; 12\bar{1}, 3,\quad [4]\, 23\bar{2};1, 3,\quad [4]\, 3; 1, 2,\quad
[4]\, 2;23\bar{2}, 2\bar{3}\bar{2}123\bar{2},
\]
and pay attention to $[6] \,\bar{3}13; 12\bar{1}, 3$. Let $Q=R_1R_2R_3.$ It is easily seen that $Q^3$ acts as a complex reflection with order 2 mapping the opposite vertices to each other. The image of it under $R_3$ is $[6]\,1; 3, 12\bar{1}.$ Now, we firstly consider the pentagon lying in $L_1$ comprising the triangle from $[3]\, 1; 2, 3$ and a quadrilateral which is half of the hexagon from $[6]\, 1; 3, 12\bar{1}.$ 
It has vertices $x_j={\bf{P}}(\n_1\boxtimes \a_j)$ (see Figure 3), where 
%\begin{figure}[htp]\label{fig:F.D.S2}
%\includegraphics[height=2.5in]{F.D.S2.eps}
%\caption{The nonagon $F$}
%\end{figure}
\begin{alignat*}{4}
\a_1=31\n_2, \,
\a_2=\bar{2}\bar{1}\n_3, \,
\a_3=\n_2,\, 
 \a_4=2\n_3,\,
 \a_5=\n_3.
 \end{alignat*}
Note that  $\a_0$ is the polar vector of $3Q^3\bar{3}$ which has fixed point lying in the  geodesic spanned by $x_1, x_2.$  It is a simple matter to check that the following side pairing transformations
\begin{align*}
g_1=3Q^3\bar{3}:&\quad x_0\longmapsto x_0,\quad x_1\longmapsto x_2,\\
g_2=(12)^2:&\quad x_3\longmapsto x_3,\quad x_2\longmapsto x_4,\\
g_2=(13)^2:&\quad x_5\longmapsto x_5,\quad x_4\longmapsto x_1.
\end{align*}
We get the $\C$-Fuchsian group fixing $L_1$ has the presentation
\[
\langle g_1, g_2, g_3: g_1^2, g_2^{|\frac{2p}{p-4}|}, g_3^{|\frac{2p}{p-4}|}, (g_2\circ g_1^{-1}\circ g_3)^{|\frac{2p}{p-4}|}\rangle.
\]

We claim that there exist $\C$-Fuchsian subgroups fixing complex geodesic $L_3$ which are obviously not conjugate to the ones stated above.  
We consider the three pyramids with quadrilateral bases and make each of them have the base in $L_3:$
\begin{align*}
&[4]\, 3; 1, 2\\
&[4]\, 23\bar{2};1, 3\stackrel{\bar{2}}{\longrightarrow}
[4] \, 3; \bar{2}12, 32\bar{3}\\
&[4]\, 2;23\bar{2}, 2\bar{3}\bar{2}123\bar{2}\stackrel{23}{\longrightarrow}
[4]\, 3; 2, 31\bar{3}.
\end{align*}
We glue the three quadrilaterals lying in $L_3$ and get an octagon (see Figure 6) with vertices $y_j={\bf{P}}(\n_3\boxtimes \b_j)$, where
\begin{figure}[htp]\label{fig:F.D.E2}
\includegraphics[height=2.5in]{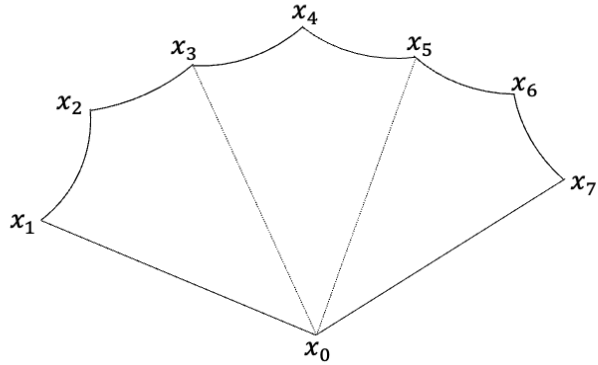}
\caption{The octagon $F$}
\end{figure}
\begin{alignat*}{4}
&\b_0=\n_2, \,
\b_1=23\n_1, \,
\b_2=\bar{1}\bar{3}\n_2,\, 
 \b_3=\n_1,\\
 &\b_4=1\n_2,
 \b_5=\bar{2}\n_1,\,
 \b_6=\bar{2}1\n_3,\,
 \b_7=\bar{2}\bar{3}\bar{2}\n_1.
 \end{alignat*}
 Let the sides denote by
\begin{equation}\label{sig4-s}
l_{i}=\left\{
             \begin{array}{ll}
             {\bf{P}}({\rm{Span}}_\C\{x_{i-1}, x_i\})\cap L_3, &\quad i=1, 2, \cdots 7.\\
             \\
             {\bf{P}}({\rm{Span}}_\C\{x_{7}, x_0\})\cap L_3,  &\quad i=8.  
             \end{array}
\right.
\end{equation}
One can check that the side pairing transformations
\[
h_1=(23)^3,\quad h_2=(13)^2,\quad h_3=(\bar{2}123)^2,\quad h_4=(231\bar{3}\bar{2}3)^2\circ h_1
\]
satisfy
$$h_1(l_8)=l_1,\quad h_2(l_3)=l_4,\quad h_3(l_5)=l_6,\quad h_4(l_7)=l_2.$$
In particular, we give an explanation for $h_4(l_7)=l_2.$ Note that $131\n_3=\n_3$ and 
$231\bar{3}\n_2=3\bar{1}\bar{3}\n_2$ hold due to $1313\bar{1}\bar{3}\bar{1}=3$ and ${\rm{br}}_3(1, 23\bar{2}).$ Then we have 
\begin{align*}
(231\bar{3}\bar{2}3)^2(23)^3(\bar{2}1\n_3)&=3^2(231\bar{3}\bar{2}3231\bar{3}\bar{2})(23232)(\bar{2}1\n_3)\\
&=3^2231(\bar{3}\bar{2}323)1231\n_3\\
&=3^223(1212)31\n_3\\
&=3^223212(131\n_3)\\
&=3^3231\bar{3}\n_2\\
&=3^33\bar{1}\bar{3}\n_2\\
&=3^4\bar{1}\bar{3}\n_2,
\end{align*}
which indicates that $h_4(y_6)=y_2.$ One can check the elliptic cycle at $y_4$ satisfies 
$h_2\circ h_4\circ h_3=3^6(\bar{3}1\bar{2}\bar{1}),$ then obtain the presentation of $\C$-Fuchsian group fixing $L_3:$ 
\[
\langle h_1, h_2, h_3: h_1^{|\frac{2p}{p-6}|},\,h_1^{|\frac{2p}{p-4}|}, h_3^{|\frac{2p}{p-4}|},\, (h_4\circ h_1^{-1})^{|\frac{2p}{p-4}|}, (h_2\circ h_4\circ h_3)^{|\frac{p}{p-3}|} \rangle.
\]
%\noindent \emph{
%Remark 1.} For Thompson group $\textbf{E}_2$, we also could get lattices in complex geodesic $L_3$ which are not conjugate to the $\C$-Fuchsian subgroups stated in Theorem \ref{thm:thom2}. Let $A=(3\bar{2}\bar{3}132)^2$, $B=(3\bar{2}32)^3$, $C=(3\bar{2}12)^2$ and $D=(3\bar{2}\bar{3}123\bar{2}\bar{1}32)^3\circ A\circ C$. Then the explicit representations are as follows.\\
%(\romannumeral1) $p=3$ $$\langle A, B, C, D: A^6, B^2, C^6\rangle,$$
%where $A, B, C$ are complex reflections in a point respectively, 
%(just have one fixed point in $\HC$)
%$D$ is a loxodromic element.\\
%(\romannumeral2) $p=4$ $$\langle A, B, C, D: B^4\rangle,$$
%where $A$ and $C$ are unipotent parabolics, $B$ is a complex reflection in a point, 
%$D$ is a loxodromic element.\\
%(\romannumeral3) $p=6$ $$\langle A, B, C, D: A^6, C^6\rangle,$$\\
%where $A$ and $C$ are complex reflections in a complex line respectively, $B$ is a unipotent parabolic, $D$ is a loxodromic element.
%Note that each $\C$-Fuchsian subgroup above fixing $L_3$ is not conjugate to any one of the $\C$-Fuchsian subgroups in Theorem \ref{thm:thom2}.

The Fuchsian groups we investigate above are subgroups of non-arithmetic lattices acting on the complex hyperbolic plane. One can immediately get that: if $\Gamma$ is a lattice and $A\in \Gamma$ is a complex reflection fixing a complex geodesic $L_A$, then {\bf Stab}$_\Gamma(L_A)$ intersects {\bf Stab}$_{{\SU}(2, 1}(L_A)$ in a lattice. Then one can easily get the following proposition.

\begin{prop}\label{thm:lat}
There exist lattices in {\bf Stab}$_{\SU(2, 1)}(L_1)$, which could be embedded in $\SU(1, 1)$. They are subgroups of the complex hyperbolic triangle groups, which we considered in Theorem \ref{thm:main}. %Theorem \ref{thm:thom1} and Theorem \ref{thm:thom2}.
\end{prop}

\centerline{\bf{Acknowledgements}}
\noindent
The author would like to thank Martin Deraux for drawing her attention to the topic of complex hyperbolic lattices and several valuable discussions. The author is also grateful to Ioannis Platis, Toshiyuki Sugawa for providing helpful comments and suggestions.  The author wishes to express her thanks to the referees for their constructive comments which substantially helped improving this paper.
%This work is supported by China Postdoctoral Science Foundation No. 2020M672036.
 
\renewcommand \refname{\bf{References}}
  \bibliographystyle{abbrv}
\bibliography{sun-bib.bib}
  \end{document}